\documentclass[12pt]{amsart}
\usepackage[centertags]{amsmath}
\usepackage{fullpage}
\usepackage{amsfonts}
\usepackage{amsthm}
\usepackage{newlfont}
\usepackage{amscd}
\usepackage{amsgen}
\usepackage{longtable}
\newlength{\defbaselineskip} 
\theoremstyle{plain}
\newtheorem{thm}{Theorem}[section] 
 \newtheorem{prop}[thm]{Proposition}
\newtheorem{lemm}[thm]{Lemma}

\theoremstyle{remark} \newtheorem{rem}[thm]{Remark}
\theoremstyle{definition} \newtheorem{defi}[thm]{Definition}
\theoremstyle{definition} \newtheorem{ex}[thm]{Example}
\theoremstyle{definition} 

 \numberwithin{equation}{section}

\begin{document}
\author{Micha\l\ Kapustka}

\title{Geometric transitions between Calabi--Yau threefolds related to
Kustin-Miller unprojections}

\begin{abstract}
 We study Kustin-Miller unprojections between Calabi--Yau threefolds, or more
precisely the geometric transitions they induce. We use them to connect many
families of Calabi--Yau threefolds with Picard number one to the web of Calabi-
Yau complete intersections. This result enables us to find explicit description of a
few known families of Calabi--Yau threefolds in terms of equations. Moreover, we
find two new examples of Calabi--Yau threefolds with Picard group of rank one,
which are
described by Pfaffian equations in weighted projective spaces.
\end{abstract}

\maketitle
\begin{section}{Introduction} In this paper a Calabi--Yau threefold $X$ is a
smooth projective threefold with trivial canonical divisor and vanishing
cohomology $H^1(X,\mathcal{O}_X)=H^2(X,\mathcal{O}_X)=0$. Such varieties are
intensely investigated in particular because of the crucial role they play in
superstring theory. The simplest example of such a threefold is a smooth quintic
in $\mathbb{P}^4$. There are however much more examples, making the general
theory of Calabi--Yau threefolds very complicated. One of the reasons is that
even the simplest example (for instance the quintic) may admit many
degenerations whose resolution will again be a Calabi--Yau threefold. Such an
operation is called a geometric
transition. Restricting the singularities 
of the degenerate variety to be nodes,
one studies so-called conifold transitions. In \cite{Str}, Strominger gave an
interpretation 
of conifold transitions in the context of black hole condensations. Nowadays a
few more geometric
transitions have found a good interpretation in terms of
superstring theory (see \cite{KMP,BKK,BKKM}). 

In \cite{Gr2} Gross used geometric transition to formulate a version of the
so-called ``Reid fantasy'' \cite{RF} 
which doesn't need to go beyond the algebraic category. He conjectured that any
two Calabi--Yau threefolds might be connected by
sequences of geometric transitions. In Physical terms this statement is called
the Web conjecture, as it suggests that all Calabi--Yau threefolds
form a giant web, whose connections are geometric transitions.

So far it has been proven that all Calabi--Yau threefolds which are
complete intersections in products of projective spaces are connected using
geometric transitions (see \cite{CDLS}). 
Moreover, a large class of hypersurfaces in weighted 
projective four-space has also been connected to the web (see
\cite{CLyS,CGGK}). 
In the latter context the candidates for the transitions are constructed, very
naturally, by
intersecting the polytopes defining the Calabi--Yau varieties, and checking
whether the intersection is also a reflexive polytope. This construction in
geometrical terms may be interpreted as a composition of projections of the
degenerate Calabi--Yau threefold completed by a smoothing.

The aim of this paper is to study geometric transitions related with projections
(or equivalently but more adequately in our context with unprojections). More
precisely, under some additional assumptions, a
projection will induce a pair of geometric transitions which we shall call, to
avoid confusion, a geometric bitransition. 
This class of bitransitions, thanks 
to its naturality, may be much better understood, and
hence gives us a new variety of tools that may be used for explicit descriptions
of examples of Calabi--Yau threefolds. The theory of unprojections has already
proved its efficiency in classical algebraic geometry.
Thanks to works of Reid and Papadakis \cite{Reid 2,Pap,RP,Pap2,Pap3},
unprojections became a powerful tool to
describe and construct new varieties.
They are used in descriptions of singular K3 surfaces and Fano threefolds. In
this paper we would like to show how these tools work in the context of
Calabi--Yau threefolds. Differently from the approach to K3 surfaces or Fano 3
folds, we will not be interested in the singular varieties arising by
unprojection, but in the families of smooth
Calabi--Yau threefolds that might degenerate to such. The construction
shall give rise to geometric bitransitions between these families. Throughout
the
paper we study examples of such given geometric bitransitions. In this way we
find connections 
(consisting of an even number of geometric transitions)
between many known families of Calabi--Yau
threefolds with Picard number one. In particular we connect to the web the 5
families 
of Calabi--Yau threefolds introduced by Borcea \cite{Borc} and described as
complete
intersections in homogeneous spaces.   

The main advantage of our constructions is that at least in low codimension they
are well understood in terms of equations.
Thanks to that, for some known examples of Calabi--Yau threefolds we find
descriptions in terms of equations in weighted projective space. This fact enables us
to understand much better the geometry of those examples of Calabi--Yau
varieties,
which were known only as smoothings of some degenerate varieties. 

Moreover, in two of the cases the construction leads us to new Calabi--Yau
threefolds. They are
described by the vanishing of Pfaffians of some $5\times 5$ matrix with
polynomial
entries in some weighted projective spaces. These two families, together
with a few others explicitly described in this paper, have
conjectured Picard-Fuchs equations of their mirror (see \cite{TableVS}). We hope
one might use the results of this paper to find these mirrors. There are at
least two possibilities to proceed. The first is to use directly the constructed
geometric transitions to find candidates for the mirrors. This idea is based on
the conjecture by Morrisson (see \cite{Morglass}) which states that any
geometric transition
should admit a mirror transition going in the opposite direction. The second
would be based on using the explicit descriptions found to proceed
as in \cite{B}. In most cases one could also try to mimic the ideas of 
\cite{Rod} and \cite{Tjo}. The two new examples are additionally interesting
because the conjectured
 Picard-Fuchs equations of their mirrors admit
non-integral elliptic instanton numbers.

\end{section}
\begin{section}{Preliminaries}
Although as announced in the introduction throughout the paper the term Calabi--Yau threefold shall
concern only smooth projective manifold, for clarity of presentation we shall use also the term singular Calabi--Yau threefold as follows.
\begin{defi}
A normal Gorenstein variety $X$ is called a singular Calabi--Yau threefold if it has trivial dualizing 
sheaf and satisfy $H^1(X,\mathcal{O}_X)=H^2(X,\mathcal{O}_X)=0$. 
\end{defi}

Geometric transitions studied in this paper are defined as follows.
\begin{defi}
A geometric transition from a Calabi--Yau threefolds $X$ to a Calabi--Yau
threefold $Y$ is a pair consisting of a birational morphism 
$f:X\longrightarrow Z$ and a flat family over a disc with central fiber $Z$ and
some other fiber $Y$. 
Where $Z$ is a possibly (preferably) singular Calabi--Yau threefold. In this context
the latter deformation will be
called a smoothing of $Z$ or a degeneration of $Y$ depending on the direction
from which we look.     
\end{defi}

\begin{defi}
Two families $\mathcal{X}$, $\mathcal{Y}$ of Calabi--Yau threefolds are
connected by a geometric transition, if there exist
smooth elements $X \in \mathcal{X}$ and $Y\in \mathcal{Y}$ that are connected by
a geometric transition.   
\end{defi}

\begin{ex} A standard example is the following. Let $Z$ be a general quintic
containing a plane. It is a degeneration of 
a smooth quintic $Y$ which has 16 nodes. The blow up $X$ of $Z$ in the plane is
a smooth Calabi--Yau threefold. Then the considered 
blow up together with the deformation is a geometric transition connecting $X$
and $Y$.
\end{ex}
\begin{defi}
We shall say that two Calabi--Yau threefolds are connected by a geometric
bitransition if there exists a Calabi--Yau threefold $Z$ 
such that there are geometric transitions both from $Z$ to $X$ and from $Z$ to
$Y$.   
\end{defi}

\begin{ex} A standard example of bitransition is taken from \cite{Gr2}, and is the following. Consider a
quintic with a triple point. Blowing up this triple 
point is a resolution of singularities with a exceptional divisor a del Pezzo
surface of degree 3. The second extremal ray of the 
obtained variety $Z$ defines a map on a double octic threefold contracting
61 lines to nodes.  In this way we connect by a geometric bitransition the family
of quintic Calabi--Yau threefolds with the family of double octic Calabi--Yau threefolds (compare with case
15 in the Table of subsection \ref{sec codim 2}).
\end{ex}

\begin{rem}
Observe that with our definition two families of Calabi--Yau threefolds with
Picard number one cannot be connected by 
one geometric transition. This is because smooth varieties with Picard number
one admit no contraction morphisms. That is why we 
introduced the notion of bitransition as being the most natural way to connect
two Calabi--Yau threefold of Picard number one. 
One could alternatively consider a weaker definition of geometric transition
without assuming $f$ to be a morphism. Then two 
Calabi--Yau threefolds with Picard number one could a priori be connected by
such transitions, but unfortunately 
no such example is known to the author.      
\end{rem}

For a detailed introduction to the theory of geometric transitions see
\cite{Ros}.
Let us now recall some basic facts about Kustin--Miller unprojections (for more
details see \cite{Reid 2,RP}).
An unprojection is a birational map which is inverse to some projection.
We shall consider unprojections arising in the following construction.

Let  $D\subset X\subset \mathbb{P}^n$ be two projectively Gorenstein varieties
such that $D$ has codimension 1 in $X$.  
Let $\omega_X=O_X(k_X)$ and $\omega_D=O_D(k_D)$ be the dualizing sheaves of $X$
and $D$. Assume that $k_X>k_D$. 
Then by \cite[Thm 5.2]{Reid 2} there exists a section $s\in O_X(k_x-k_D)$ with
poles along $D$ which defines a 
birational map:
$$\varphi: X\dashrightarrow Y \subset\mathbb{P}^n[s]= \operatorname{Proj}
\mathbb{C}[x_0,\dots,x_n,s]$$  
 contracting $D$ to the point $P_s:=(0:\dots:0:1)$. Moreover $Y$ is also
projectively Gorenstein.
The inverse of $\varphi$ is the projection of $Y$ from the point $P_s$. 
\begin{defi}
The map $\varphi$ obtained above will be called a Kustin Miller unprojection
with exceptional divisor $D$ and unprojected variety $X$. 
\end{defi}

\begin{rem}
The above construction doesn't need the ambient space $\mathbb{P}^n$. In fact it
is often performed without any 
ambient space. In our context it will be convenient to have an ambient space
$T=\operatorname{Proj} R$ and then 
$Y\subset \operatorname{Cone}(T):=\operatorname{Proj} R[s]$     
\end{rem}
\begin{rem}\label{factorizes}
Observe that the projection inverse to $\varphi$ always factorizes through the
blow-up $Z$ of $P_s$ and a contraction morphism.
\end{rem}

\begin{ex}\label{Ax-By}
The standard example of Kustin--Miller unprojection is based on the famous Reid
$Ax-By$ trick. 
Let us recall it here following \cite[sec.2]{Reid 2}.
Let $T=\mathbb{P}^n(a_0,\dots,a_n)$ with coordinate system $x=(x_0,\dots,x_n)$.
Let $D=\{x\in T| A(x)=B(x)=0\}$ 
for some homogeneous polynomials $A$ and $B$ of degrees $d_1$ and $d_2$, and let
$X$ be a hypersurface containing $D$.
Then $X=\{x\in T| A(x)C(x)-B(x)D(x)=0\}$ with $C$ and $D$ homogeneous
polynomials of degrees $k_1$,$k_2$ 
such that $k_1+d_1=k_2+d_2$.   
Assume that $d_1>k_2$. Then we can define $Y=\{x\in
\operatorname{Proj}(\mathbb{C}[x_0,\dots,x_n,s])| C(x)s=A(x), D(x)s=B(x) \}$. 
Observe that $Y$ contains the point $P_s=(0:\dots:0:1)$, and that the image of
$Y\setminus \{P_s\}$ by the projection from
$P_s$ contains $X\setminus D$.  
\end{ex}

In this paper we shall consider the following construction
based on Kustin--Miller unprojections of del Pezzo surfaces in Calabi--Yau
threefolds.  Take a del Pezzo surface $D$
anti-canonically embedded in a (weighted) projective space, i.e.
in such a way that the restrictions of the hyperplanes from the
projective space is the complete anti-canonical system ($D$ may be
contained in some hyperplane). Consider moreover a proper flat family of
Calabi--Yau threefolds with Picard number one embedded in the same
projective space by the generators of their Picard groups. Assume that 
the elements of this family are projectively Gorenstein. Suppose
that in the considered family we find a singular Calabi--Yau threefold $X$ 
containing the surface $D$. As both $D$ and $X$ are projectively
Gorenstein, by Kustin-Miller there is a rational section $s\in O_X
(1)$ with poles along $D$. The closure of the graph of this
section (in the cone $\operatorname{Cone}(T)$) is a variety $Y$ birational to
$X$, and which is a singular Calabi--Yau threefold with isolated
singularity isomorphic to a cone over the anti-canonically embedded del Pezzo
surface $D$. For most del Pezzo surfaces (except for the Hirzebruch surface
$\mathbb{F}_1$) under some additional assumptions (see \cite{Gr2,namikawa1})
such singular
variety $Y$ can be smoothed to a Calabi--Yau threefold. Moreover, the smoothing
may then be performed in the same space
in which the singular variety is embedded. In fact, in our context we won't need
the existence of the smoothing as we will have an explicit description of the
singular variety leading us to a natural explicit smoothing.  As $X$ can also be
smoothed by assumption we obtain in this way a data consisting of a
degeneration 
followed by a birational map and a smoothing. Moreover, our birational map can
be factorized as in Remark \ref{factorizes}. Now provided that $Z$ 
is a Calabi--Yau threefold we obtain a geometric bitransition between two
families of Calabi--Yau threefolds.

In order to use the above construction to find explicit examples of geometric
transitions, we
first need to find a suitable embedding of a del Pezzo surface into a singular
Calabi--Yau threefold admitting a smoothing. Then, after having performed the unprojection, we need to
find an explicit smoothing of the obtained singular variety. To complete the
picture properly, we shall also need 
a crepant resolution of each of the considered singular varieties.

One can use the following lemma to construct a table of candidates
for the considered construction in terms of basic invariants.
Assume that we have two families $\mathcal{X}$ and $\mathcal{Y}$ of embedded
Calabi--Yau threefolds connected by an unprojection in the above sense.
Let $X$ and $Y$ be general elements of these families and $H_X$ and $H_Y$ their
respective hyperplane sections.
Let $\pi\colon X_0\rightarrow Y_0$ be the unprojection connecting the families.
It is a birational map between two singular Calabi--Yau threefolds,
$X_0\in\mathcal{X}$ and $Y_0\in \mathcal{Y}$ contracting a Del
Pezzo surface $D$ of degree $d$.
\begin{lemm} \label{niezmienniki} Under the above assumptions the invariants of
$X$ and $Y$ are
related by the following formulas:
\begin{itemize}
\item[(i)]$\dim (|H_Y|)=\dim (|H_X|)+1$
 \item[(ii)] $H_Y^3=H_X^3-d$
\item[(iii)] $c_2(Y).H_Y= c_2(X).H_X -12+2d$
\end{itemize}
\end{lemm}
\begin{proof} Item (i) follows directly from the assumption. Item (ii) follows
from the fact that the preimage of a general linear section of
$X_0$ by $\pi^{-1}$ is a linear section of $Y_0$ passing through
the center of the projection $\pi^{-1}$ and that the degree is
equivariant under flat deformations. Item (iii) follows from the
Riemann-Roch theorem for threefolds.
\end{proof}

\begin{rem}
Observe that the invariants $H^3$ and $c_2.H$ are enough to determine whether
there exists a flat deformation between two Calabi--Yau threefolds (see
\cite[Thm. 7.4]{Lee}).
\end{rem}

\begin{rem}
The constructed table of candidates should be considered as giving us only hints
about what examples to look for. We will see however that it is very incomplete
for two reasons. The first is that: even if a pair of families appears to be a
good
candidate basing only on the invariants, it does not necessarily mean that we
can
directly construct a suitable unprojection between them. The second is that the
list of Calabi--Yau threefolds with Picard group of rank one is far to be
complete. Hence even if a Calabi--Yau threefold appears not to admit any
unprojection that connects it to other families using geometric transitions, it
does not mean that there is no connection at all which uses these constructions.
This
means only that we need an intermediate family that is not included in the
table. Moreover, as by unprojection we do not always end up in Calabi--Yau
threefolds with Picard group of rank one, this intermediate family needs not to
be of Picard number one.
\end{rem}

The constructions involving unprojections are very convenient tools to describe
already known varieties in terms of equations. For instance the following
proposition shows that most constructions contained in \cite{GMP1,G2,GMk} are in
fact unprojections in the sense studied in this paper.
\begin{prop}\label{projekcja jest zadana przez D+H}
Let $X\subset\mathbb{P}^n$ be a nodal Calabi--Yau threefold
containing a del Pezzo surface $D$ in its anti-canonical
embedding. Let $H$ be the hyperplane section of $X$. Let
$\tilde{X}$ be a small resolution of $X$ arising by blowing up $D$
and flopping the obtained lines. Then the morphism given by the
system $|\tilde{D}+\tilde{H}|$ on $\tilde{X}$ is the composition
of the blowing down with the unprojection map contracting $D$.
\end{prop}
\begin{proof}
Observe that the morphism  $\varphi_{|\tilde{D}+\tilde{H}|}$ given
by the system $|\tilde{D}+\tilde{H}|$ on $\tilde{X}$ contracts
$\tilde{D}$. Denote the image of $D$ by $p$.  Let us consider the
composition of the morphism given by $|\tilde{D}+\tilde{H}|$ on
$\tilde{X}$ with the projection from $p$. This rational morphism
is given outside $\tilde{D}$ by the subsystem of
$|\tilde{D}+\tilde{H}|$ given by divisors with $D$ as component.
This implies that the considered composition is given outside $D$
by the system $\tilde{H}$, i.e. it is the blowing down of the
small resolution. We hence obtain the equality from the assertion
outside $D$ which implies equality everywhere.
\end{proof}

\end{section}

\begin{section}{Classical unprojections}
The classical examples of unprojections for which the result is
well described using equations concern the situation where the
unprojected variety (in our case $D$) is a complete intersection
of low codimension.

\subsection{codimension 2} \label{sec codim 2} Del Pezzo surfaces which are a codimension 2 complete
intersections in some weighted projective space can be naturally
embedded into singular Calabi--Yau threefolds which are
hypersurfaces in the same space. This is done in the following way.
Assume that $D$ is given by equations $q_1$,$q_2$ of (weighted) degrees $d_1$,
$d_2$
in some space $T$. Let $\mathcal{X}$ be the family of hypersurfaces of degree
$d > \operatorname{max}\{d_1,d_2\}$ in $T$. Then the variety $X_0$ given by the
equation $q_1 a_1+q_2 a_2=0$
is an element of $\mathcal{X}$ that contains $D$ for any $a_1$, $a_2$ of degrees
$d-d_1$, $d-d_2$ respectively. We are then in position of Example \ref{Ax-By}. 
Hence the Kustin--Miller unprojection corresponding to our data leads to a
codimension
2 complete intersection $Y_0$ in a cone over our starting space. 
As we observed in Remark \ref{factorizes} the inverse birational map factorizes
through 
the blow up $Z$ of the vertex $P$ of the cone and a birational contraction
morphism.
In the considered case for general $X_0$ the contraction morphism contracts
proper transforms of lines given by
$q_1(x)=a_1(x)=q_2(x)=a_2(x)=0$ in $\operatorname{Cone}(T)$ to nodes given by
the same equations in $T$. 
Now $Z$ is a crepant resolution of both $X_0$ and $Y_0$. The first being the
blow up of a singularity 
with tangent cone a del Pezzo surface the latter a small resolution of nodes. It
follows that $Z$ is a Calabi--Yau threefold 
(in particular it is smooth) connected by a geometric transition to any smooth
representative of $\mathcal{X}$ and any smoothing of $Y_0$.
In each of the cases below the smoothing of $Y_0$ is explicit and hence the
obtained pair of geometric transitions is an explicit geometric 
bitransition related to a Kustin-Miller unprojection.       
\begin{center}
\begin{longtable}{c|c|c|c}
$T$&del Pezzo & Calabi--Yau & result of unprojection\\
\hline \hline \endhead
 $T_6 \subset \mathbb{P}(1^2,2^2,3^2)$&$D_{2,3,6}\subset
\mathbb{P}(1^2,2^2,3^2)$&$X_{6,6}\subset
 \mathbb{P}(1^2,2^2,3^2)$&$Y_{3,4,6}\subset\mathbb{P}(1^3,2^2,3^2)$ \\ \hline
$T_6 \subset \mathbb{P}(1^3,2^2,3)$&$D_{1,2,6}\subset
\mathbb{P}(1^3,2^2,3)$&$X_{4,6}\subset\mathbb{P}(1^3,2^2,3)$&$Y_{2,3,6}
\subset\mathbb{P}(1^4,2^2,3)$\\
\hline $T_6 \subset \mathbb{P}(1^4,2,3)$& $D_{1,1,6}\subset
\mathbb{P}(1^4,2,3)$&$X_{3,6}\subset\mathbb{P}(1^4,2,3)$&$Y_{2,2,6}
\subset\mathbb{P}(1^5,2,3)$\\
\hline $T_4\subset \mathbb{P}(1^3,2^2,3)$&
$D_{4,2,3}\subset\mathbb{P}(1^3,2^2,3)$&
$X_{4,6}\subset\mathbb{P}(1^3,2^2,3)$&
$Y_{4,4,3}\subset\mathbb{P}(1^4,2^2,3)$\\
\hline $T_4\subset \mathbb{P}(1^4,2^2)$&
$D_{1,2,4}\subset\mathbb{P}(1^4,2^2)$&
$X_{4,4}\subset\mathbb{P}(1^4,2^2)$&
$Y_{2,3,4}\subset\mathbb{P}(1^5,2^2)$\\
\hline $T_4\subset \mathbb{P}(1^5,2)$&
$D_{1,1,4}\subset\mathbb{P}(1^5,2)$&
$X_{3,4}\subset\mathbb{P}(1^5,2)$&
$Y_{2,2,4}\subset\mathbb{P}(1^6,2)$\\
\hline $\mathbb{P}(1^4,2)$& $D_{2,3}\subset\mathbb{P}(1^4,2)$&
 $X_6\subset \mathbb{P}(1^4,2)$&
 $Y_{3,4}\subset \mathbb{P}(1^5,2)$\\
\hline $T_3\subset \mathbb{P}(1^5,2)$&
$D_{1,2,3}\subset\mathbb{P}(1^5,2)$&
 $X_{3,4}\subset \mathbb{P}(1^5,2)$&
 $Y_{2,3,3}\subset \mathbb{P}(1^6,2)$\\
\hline $T_3\subset \mathbb{P}^5$&
 $D_{1,1,3}\subset\mathbb{P}^5$&
$X_{3,3}\subset \mathbb{P}^5$&
 $Y_{2,2,3}\subset \mathbb{P}^6$\\
\hline $T_2\subset \mathbb{P}(1^5,3)$&
$D_{2,2,3}\subset\mathbb{P}(1^5,3)$&
 $X_{2,6}\subset\mathbb{P}(1^5,3)$&
$Y_{2,3,4}\subset \mathbb{P}(1^6,3)$\\
\hline $T_2\subset \mathbb{P}^5$&
 $D_{1,2,2}\subset\mathbb{P}^5$&
 $X_{2,4}\subset\mathbb{P}^5$&
$Y_{2,2,3}\subset \mathbb{P}^6$\\
\hline $T_{2,2}\subset \mathbb{P}^6$&
$D_{1,1,2,2}\subset\mathbb{P}^6$&
 $X_{2,2,3}\subset \mathbb{P}^6$&
 $Y_{2,2,2,2}\subset \mathbb{P}^7$\\
 \hline
$Pf\subset\mathbb{P}(1^7,2)$&
 $D_{1,2}\cap Pf \subset\mathbb{P}(1^7,2)$&
 $X_{4}\cap Pf \subset\mathbb{P}(1^7,2)$&
$Y_{2,3}\cap Pf \subset\mathbb{P}(1^8,2)$\\
\hline $Pf \subset\mathbb{P}^7$&
 $D_{1,1}\cap Pf \subset \mathbb{P}^7 $&
$X_{3}\cap Pf \subset \mathbb{P}^7$&
 $Y_{2,2}\cap Pf \subset \mathbb{P}^8$\\
\hline $\mathbb{P}(1^4,4)$&
 $D_{3,4}\subset \mathbb{P}(1^4,4)$ &
$X_{8}\subset \mathbb{P}(1^4,4)$&
 $Y_{4,5}\subset \mathbb{P}(1^5,4)$\\
\hline $\mathbb{P}(1^3,2,5)$&
 $D_{4,5}\subset \mathbb{P}(1^3,2,5)$ &
$X_{10}\subset \mathbb{P}(1^3,2,5)$&
 $Y_{5,6}\subset \mathbb{P}(1^4,2,5)$\\
 \hline $T_6\subset\mathbb{P}(1^4,2,3)$&
 $D_{1,1,6}\subset \mathbb{P}(1^4,2,3)$ &
$X_{3,6}\subset \mathbb{P}(1^4,2,3)$&
 $Y_{2,2,6}\subset \mathbb{P}(1^5,2,3)$\\
\hline $T_2\subset \mathbb{P}(1^5,3)$&
 $D_{2,2,3}\subset \mathbb{P}(1^5,3)$ &
$X_{2,6}\subset \mathbb{P}(1^5,3)$&
 $Y_{2,4,3}\subset \mathbb{P}(1^6,3)$\\
 \hline $T_4\subset \mathbb{P}(1^4,2)$&
 $D_{2,3}\subset \mathbb{P}(1^4,2)$ &
$X_{6}\subset \mathbb{P}(1^4,2)$&
 $Y_{3,4}\subset \mathbb{P}(1^5,2)$\\
\hline $\mathbb{P}^4$&
 $D_{3,1}\subset \mathbb{P}^4$ &
$X_{5}\subset \mathbb{P}^4$&
 $Y_{2,4}\subset \mathbb{P}^5$\\
\hline $\mathbb{P}^4$&
 $D_{2,2}\subset \mathbb{P}^4$ &
$X_{5}\subset \mathbb{P}^4$&
 $Y_{3,3}\subset \mathbb{P}^5$\\

\end{longtable}
\end{center}
The notation $Pf$ in the table stands for the variety given by $4 \times 4$
Pfaffians of a $5\times5$ matrix with general linear entries in the appropriate
space (it describes either a linear section or a cone over the Grassmannian
G(2,5) in it's Pl\"ucker embedding). Here all the varieties obtained as results
of the unprojections in the table are well known to admit smoothings, as the
general members with given description are known to be classical examples of
smooth Calabi--Yau threefolds.
\begin{rem} \label{nody} One can also use methods of \cite{G2} or direct
computation
with Macaulay 2 \cite{M2}, that
in all above cases the unprojection factorizes into a small resolution of nodes
and a primitive contraction of the del Pezzo surface as described in Lemma
\ref{projekcja jest zadana przez D+H}. The latter factorization is the same as
the one from Remark \ref{factorizes}.
\end{rem}

\subsection{codimension 3} Del Pezzo surfaces which are a
codimension 3 complete intersections in some weighted projective
space can be naturally embedded into appropriate singular Calabi--Yau
threefolds which are codimension 2 complete intersections. By
\cite{RP} the unprojection corresponding to such data leads us to
a variety given by $4 \times 4$ Pfaffians of a $5\times 5$ matrix with entries
of appropriate weight.
More precisely let $D$ be given by the equations $q_1$,$q_2$,$q_3$ of degrees
$d_1$,$d_2$,$d_3$ respectively. Let $\mathcal{X}$ be a family of Calabi--Yau
threefolds given as complete intersections of hypersurfaces of degrees $e_1$,
$e_2$ , where $\operatorname{min}\{e_1, e_2\}>\operatorname{max}\{d_1, d_2, d_3
\}$. Let $X_0$ be given by equations $q_1 a_1+q_2 a_2+q_3 a_3=0$, $q_1 b_1+q_2
b_2+q_3 b_3=0$, where for all $i\in\{1,3\}$ we have $a_i$, $b_i$ are generic of
degree $e_1-d_i$, $e_2-d_i$ respectively. The variety $Y_0$ given in the cone
over
$T$ by the Pfaffians of the matrix
\begin{displaymath}
\left( \begin{array}{ccccc}
 0& t & a_1 & a_2 & a_3 \\
-t& 0 & b_1 & b_2 & b_3 \\
-a_1& -b_1 & 0 & q_3 & -q_2 \\
-a_2& -b_2 & a_1 & 0 & q_1 \\
-a_3& -b_3 & a_1 & a_2 & 0
 \end{array}\right)
\end{displaymath}
is then a variety whose projection from $t=1$ (all other coordinates zero) is
$X_0$ and such that the exceptional divisor of the projection is given by
$q_1=0$,$q_2=0$,$q_3=0$. Here in the same way as before the blow up $Z$ of the
singular point is a crepant resolution of $Y_0$ and 
the morphism induced on $Z$ to $X_0$ is a small contraction to nodes of $X_0$.  
We hence obtain an explicit geometric bitransition in each of the following
examples.  
\begin{longtable}{c|c|c|c|c}
&$T$&del Pezzo & Calabi--Yau & result of unprojection\\
\hline \hline \endhead

\hline 1&$T_4 \subset\mathbb{P}(1^6,2)$&
 $D_{1,1,1,4}\subset\mathbb{P}(1^6,2)$&
 $X_{2,2,4}\subset\mathbb{P}(1^6,2)$&
 $Y_{4}\cap Pf \subset\mathbb{P}(1^7,2)$\\
\hline 2&$T_3 \subset \mathbb{P}^6$&
$D_{1,1,1,3}\subset\mathbb{P}^6$&
 $X_{2,2,3}\subset\mathbb{P}^6$&
 $Y_{3}\cap Pf \subset\mathbb{P}^7$\\
\hline
3&$\mathbb{P}^5$& $D_{1,2,2}\subset\mathbb{P}^5$& $X_{3,3}\subset\mathbb{P}^5$&
$WPf(1,2)\subset\mathbb{P}^6$\\
\hline 4&$T_{2,2} \subset \mathbb{P}^7$&
$D_{1,1,1,2,2}\subset\mathbb{P}^7$&
 $X_{2,2,2,2}\subset \mathbb{P}^7$&
 $Y_{2,2}\cap Pf \subset \mathbb{P}^8$\\
\hline 5&$Pf \subset \mathbb{P}^8$&
 $D_{1,1,1}\cap Pf \subset\mathbb{P}^8$&
 $X_{2,2}\cap Pf \subset \mathbb{P}^8$&
 $Pf \cap Pf\subset \mathbb{P}^9$\\
 \hline 6&$\mathbb{P}(1^5,2)$&
 $D_{2,2,2}\subset\mathbb{P}(1^5,2)$&
 $X_{3,4}\subset \mathbb{P}(1^5,2)$&
 $WPf(1,2)\subset \mathbb{P}(1^6,2)$\\
 \hline 7&$\mathbb{P}(1^4,2^2)$&
 $D_{3,2,2}\subset\mathbb{P}(1^4,2^2)$&
 $X_{4,4}\subset \mathbb{P}(1^4,2^2)$&
 $WPf(1,2,3)\subset \mathbb{P}(1^5,2^2)$\\

\end{longtable}

Let us look more precisely at the examples from the table. The result of the
unprojections in the examples 2 and 4 are well known families of Calabi--Yau
threefolds. These are complete intersections in the Grassmannian $G(2,5)$. The
example number 3 is also well known and described by $4\times 4$ Pfaffians of a
$5\times 5$ matrix with one row of quadrics and remaining entries linear. The
example 1 is the double cover of the Fano threefold $B_5$ obtained as a
codimension 3 linear section of $G(2,5)$ branched over the intersection of this
Fano variety with a quartic (see \cite{Borc}). The example number 5 was found in
\cite{G2}, it is a Calabi--Yau threefold with Picard number one. We find it's
explicit description as a complete intersection of two Grassmannians embedded in
$\mathbb{P}^9$ by different Pl\"ucker embeddings. The examples 5 and 6 are new
examples of Calabi--Yau threefolds with Picard number one. Calabi--Yau
threefolds
with such invariants have a predicted Picard-Fuchs equation of their mirror in
\cite{TableVS}.

\begin{ex}
The example 6 is described by Pfaffians of a generic antisymmetric $5\times5$
matrix with entries of the following degrees:
$$\left(\begin{array}{cccc}
1&1&1&1\\
& 2&2&2\\
&  &2&2\\
& & &2
\end{array}\right)
$$ in the weighted projective space $\mathbb{P}(1^6,2)$. To prove that these are
smooth varieties we use the computer algebra system Macaulay 2 on a specific
example. The hodge numbers can be computed directly from the above description.
We can also proceed as in
\cite[Thm 2.2]{G2}. We need only to observe that the general $X_4$ 
containing a del Pezzo surface $D_{2,2,2}\subset \mathbb{P}(1^5,2)$ is smooth
and the generic such 
$X_{3,4}\subset \mathbb{P}(1^5,2)$ has 28 nodes. 
We obtain that $h^{1,1}(Y)=1$ and $h^{1,2}=60$. Hence the Euler characteristic
$\chi(Y)=-116$.
The remaining invariants follow from Lemma \ref{niezmienniki} and are
$H_Y^3=10$, $\dim(|H_Y|)=6$ and $c_2(Y).H=52$.

\end{ex}
\begin{ex}
The example 7 is described by Pfaffians of a generic antisymmetric $5\times 5$
matrix with entries of the following degrees:
$$\left(\begin{array}{cccc}
1&1&2&2\\
& 1&2&2\\
&  &2&2\\
& & &3
\end{array}\right)
$$in the weighted projective space $\mathbb{P}(1^5,2^2)$. Observe that it is a
section of a cone (with vertex a line) 
over a weighted Grassmannian $wG(2,5)\subset \mathbb{P}(1^3,2^6,3)$ with weights
 $(\frac{1}{2},\frac{1}{2},\frac{1}{2},\frac{1}{2},\frac{3}{2})$ (see
\cite{wGr}) by a generic cubic and 4 generic quadrics.   
Analogously as before we prove that these are Calabi--Yau threefolds with Picard
number one, Euler
characteristic $\chi(Y)=-120$, $h^{1,2}=62$,$H_Y^3=7$, $\dim(|H_Y|)=5$ and
$c_2(Y).H=46$.
\end{ex}

\begin{rem}\label{uwagavs}
As we are given explicit descriptions of examples 5, 6, 7 it would be
interesting to check whether the Picard-Fuchs equations of their mirrors are
indeed those found in \cite{TableVS}. In fact examples 1 and 3 have also only
conjectured corresponding equations. One might try to proceed as in \cite{B}
using above geometric transitions instead of conifolds. What is additionally
interesting is that the Calabi--Yau equations that by \cite[Table 1.]{TableVS} 
should correspond to examples 6 and 7 (they produce the same invariants
including the Euler Characteristic)
produce some non-integral elliptic instantons $n^1_d$. This would contradict
the conjecture concerning their integrality.
\end{rem}

\begin{ex} We might have also considered the following unprojection triple.
$T_6 \subset\mathbb{P}(1^5,2,3)$, $D_{1,1,1,6}\subset\mathbb{P}(1^5,2,3)$,
 $X_{2,2,6}\subset\mathbb{P}(1^5,2,3)$.
The result of such an unprojection should be of the form
 $Y_{6}\cap Pf \subset\mathbb{P}(1^6,2,3)$.
We however observe that such a variety is always singular and it is not clear
whether it admits a smoothing.
\end{ex}

\subsection{Tom and Jerry} Del Pezzo surfaces which are
codimension 4 complete intersections can naturally be embedded
into Pfaffian threefolds in two ways called Tom and Jerry. We find two examples 
of Calabi--Yau varieties that might fit to this construction. Both are described
by Pfaffians of matrices of linear forms in some spaces $T$. Then, the del Pezzo
surfaces that we might try to unproject are given as linear sections of $T$ as
shown in the table below.
\begin{longtable}{c|c|c|c|c}
$T$&del Pezzo & Calabi--Yau & result of Tom & result of Jerry\\
\hline \hline \endhead
 $T_{2,2}\subset \mathbb{P}^8$&
 $D_{1,1,1,1,2,2}\subset\mathbb{P}^8$&
 $Pf \cap X_{2,2}\subset \mathbb{P}^8$&
 $Tom \cap Y_{2,2}\subset \mathbb{P}^9$& $Jerry \cap Y_{2,2}\subset
 \mathbb{P}^9$\\
 \hline
\hline $Pf \subset \mathbb{P}^9$&
 $D_{1,1,1,1}\cap Pf\subset\mathbb{P}^9$&
 $Pf \cap Pf\subset \mathbb{P}^9$&
 $Tom \cap Pf\subset \mathbb{P}^{10}$&
$Jerry \cap Pf \subset \mathbb{P}^{10}$\\
\end{longtable}
Both cases are in fact sections of standard constructions called original Tom
and original Jerry by different spaces $T$.
More precisely $D$ is given as a complete linear section of codimension 4 in $T$
(write $l_1,..,l_4$ for its defining linear forms). The Calabi--Yau varieties
from the family are described by $4\times 4$ Pfaffians of linear forms of a
$5\times5$ matrix with linear entries. Consider the varieties $Y_1$ given by
$4\times 4$ Pfaffians of the matrix
$$\left(\begin{array}{ccccc}
0&h_1&h_2&h_3&h_4\\
-h_1&0&0&l_1&l_2\\
-h_2&0&0&l_3&l_4\\
-h_3&-l_1&-l_3&0&0\\
-h_4&-l_2&-l_4&0&0
\end{array}\right)
$$
and $Y_2$ given by $4\times 4$ Pfaffians of the matrix
$$\left(\begin{array}{ccccc}
0&l_1&l_2&l_3&0\\
-l_1&0&0&l_3&l_4\\
-l_2&0&0&h_1&h_2\\
-l_3&-l_3&-h_1&0&h_3\\
0&-l_4&-h_2&-h_3&0
\end{array}\right).
$$
Both these varieties contain $D$. One checks easily that the result of these
standard unprojections are cones over
the Segre embeddings of $\mathbb{P}^2\times \mathbb{P}^2$ for Tom and
$\mathbb{P}^1\times \mathbb{P}^1 \times \mathbb{P}^1$ for Jerry. Hence the
results of the unprojections are intersections of these cones with the
corresponding space $T$.  It follows that in each case 
the blow up of the singularity of the variety $Y$ is a crepant resolution with
exceptional 
divisor a del Pezzo surface of degree 6.  
Observe that although these varieties admit explicit smoothings, the smoothed
varieties do not have
Picard number 1.

\end{section}

\begin{section}{Codimension 3} In all above examples we restricted to the case
where the del Pezzo surface $D$ is given as a complete intersection in the space
$T$. However, for the Kustin--Miller unprojection to exist it is only needed
that
both $D$ and $X$ are Gorenstein. Below we deal with the general case in
codimension 3. We have the following proposition.
\begin{prop} Let $D$ be given by the $n-1\times n-1$ Pfaffians of
an $n\times n$ antisymmetric matrix $M$ of linear forms in some projective
space. Let
$X$ be a complete intersection of two hypersurfaces $h_1$, $h_2$ containing
$D$. Then there is a variety $Y$ given by $n+1\times n+1$ minors
of an $n+2\times n+2$ antisymmetric matrix of linear forms, which is the
unprojection of $D$ in $X$.
\end{prop}
\begin{proof}
 Let us denote the Pfaffians of $M$ by $f_1, \dots, f_{n}$. We can then write
$h_1= a_1 f_1 + \dots + a_{n} f_{n}$, $h_2= b_1 f_1 + \dots + b_{n}
f_{n}$.
Let $t$ be the additional variable in the cone over $T$. Consider the matrix
\begin{displaymath}
N=\left( \begin{array}{ccc}
 0& t & \begin{array}{ccc} a_1 & \dots & a_{n} \end{array}\\
-t& 0 & \begin{array}{ccc} b_1 & \dots & b_{n} \end{array}\\
\begin{array}{c} a_1\\ \vdots \\ a_{n} \end{array}& \begin{array}{c} b_1\\
\vdots \\ b_{n} \end{array}&\text{\Huge{M}}
\end{array}\right)
\end{displaymath}

The $n+1 \times n+1$ Pfaffians of $N$ define a variety $Y$ whose projection from
$t=1$ (all coordinates in $T$ being 0) is the variety $X$ and the exceptional
divisor of the projection is $D$.
\end{proof}
\begin{ex} By the above we recover the construction from \cite{GMP1} of the
geometric bitransition of a del Pezzo of degree 5 in the Calabi--Yau threefold
$X_{3,3}\subset\mathbb{P}^5$, obtaining the Calabi--Yau variety given by
$6\times
6$ Pfaffians of a generic $7\times7$ antisymmetric matrix.
\end{ex}
\begin{rem}
There are also other examples were the unprojection may be performed explicitly
in a general context. For instance the constructions contained in \cite{GMk} are
Kustin-Miller unprojections with explicit descriptions of the resulting 
variety. All these may be performed in pure algebraic terms and hence be
extended to a more general context. For instance one may formulate a cascades of
unprojections involving such determinantal varieties of any size with additional
weights some of which might also give constructions of Calabi--Yau threefolds.
\end{rem}
\end{section}
\begin{section}{Linkage}
Unfortunately in higher codimension there isn't any general
description of the unprojection. However, if $D$ is a Del
Pezzo surface embedded into a complete intersection nodal
Calabi--Yau threefold $X$ we can construct a linkage between a
general hyperplane section of the variety $Y$ and the
del Pezzo surface $D$.

Indeed, we see that taking a section of $X$ given by a general
hypersurface $K$ of degree $k$ containing $D$. We obtain $X\cap K=
D\cup S$ for some surface $S\in |kH-D|$. Next, taking a generic
hypersurface $L$ of degree $K+1$ containing $S$ we get $X\cap
L=D\cup G$, where G is a surface from the system $|D+H|$. Hence by
Proposition \ref{projekcja jest zadana przez D+H} (or directly from the
existence of Kustin-Miller unprojection) it is isomorphic to a general
hyperplane section of $Y$.

The problem with the above approach is that in this way we find a
description of the resolution of the singular variety which is the
result of the unprojection. But a priori there is no reason for
the smoothed variety to have the same resolution. In the examples we dealt with
so far the resolution of the smoothing was always of the same form as the
resolution of
the result of the unprojection.
This did not follow however from a general principle but from the simplicity of
the construction which implied that we were always given a natural family with
given resolution and whose generic member is smooth.
\subsection{Unprojections of Del Pezzo surfaces of degree 6}
The case $D$ is a del Pezzo surface of degree 6 is somewhat special in this
context. This is because as it was proven in \cite{GMP1} the singular variety
obtained by such an unprojection is expected to have two different smoothing
families. We shall see that in this case although the smoothings are different,
their projective resolutions both have the same Betti numbers. Indeed, let $D$
be a del Pezzo surface of degree 6.
It is a classical fact that the del Pezzo surface $D \subset\mathbb{P}^6$ has
two different descriptions. The first sees it as a codimension two complete
linear section of the Segre embedding of $\mathbb{P}^2\times\mathbb{P}^2\subset
\mathbb{P}^8$, whilst the second interprets it as a hyperplane section of the
Segre embedding of $\mathbb{P}^1\times\mathbb{P}^1\times\mathbb{P}^1\subset
\mathbb{P}^7$. Now, instead of looking at the generic hyperplane sections of the
constructed variety $Y$, we might consider our unprojection as a hyperplane
section
of a higher dimensional unprojection involving one of these Segre embedded
products.
In order to perform the above reasoning we need to extend our Calabi--Yau
threefold $X$ as a linear section of an appropriate Fano variety. This is easily
done if the Calabi--Yau threefold is a complete intersection. We just consider
the generic complete intersection of the same kind in the higher-dimensional
space.
 We can hence perform two constructions. The first will be the unprojection of a
Fano fivefold containing $\mathbb{P}^2\times\mathbb{P}^2\subset \mathbb{P}^8$.
The second will be the unprojection of a Fano fourfold containing
$\mathbb{P}^1\times\mathbb{P}^1\times\mathbb{P}^1\subset \mathbb{P}^7$.
We see $Y$ in both pictures as a hyperplane section passing through the singular
point. Now reasoning as in \cite[Thms 2.1, 2.2]{G2} we prove that the results of
the unprojections in both cases have isolated singularities.
As the they are also irreducible their generic hyperplane sections are smooth
and define two smoothing families.
One smoothing of $Y$ will be the general hyperplane section of the constructed
fourfold, whilst the second smoothing will be a general codimension 2 linear
section of the fivefold. The smoothing inducing different local pictures of the
smoothed singularity have to be different.
Let us perform computations on an explicit example.
\begin{ex} Consider the unprojection of $D_6$ in a Calabi--Yau threefold
$X_{2,2,3}\subset \mathbb{P}^6$. We can proceed in two ways. The first is to
take the variety $D=\mathbb{P}^2\times\mathbb{P}^2\subset\mathbb{P}^8$ in it's
Segre embedding and a generic variety of type
$\bar{X}_{2,2,3}\subset\mathbb{P}^8$ containing $D$. By the above discussion the
smoothing of the result of the unprojection of $D_{6}$ in $X_{2,2,3}$ is
connected by a linkage in $\bar{X}_{2,2,3}$ with $D\cap H\subset\mathbb{P}^7$.
In particular the resolution is described by a mapping cone construction applied
twice. First, on the resolutions of $D\cap H$ and the Koszul complex of
$X_{2,2,3}\cap H\cap Q$, next on the result of this and the Koszul complex of
$X_{2,2,3}\cap H\cap C$ in the second step. The Betti diagram of the resolution
is then as follows
$$\begin{array}{ccccc}

1&0&0&0&0\\
0&2&0&0&0\\
0&8&18&8&0\\
0&0&0&2&0\\
0&0&0&0&1
\end{array}
.$$
The second possibility is to take the variety
$D=\mathbb{P}^1\times\mathbb{P}^1\times\mathbb{P}^1\subset\mathbb{P}^7$ with
$\bar{X}_{2,2,3}\subset\mathbb{P}^7$ containing $D$. Again we construct a
linkage and end up with a smoothing with resolution having the same Betti
diagram.
\end{ex}
\begin{rem}
In general it is hard to distinguish between the two smoothings of a Calabi--Yau
threefold with a singularity which is a cone over the del Pezzo surface $D_6$.
The existence of both smoothings is proved first in the local context and then
extended to the global case using formal cohomology. Here thanks to limiting
ourselves to unprojections we  extend the local Picture to the global case as
described above. Both smoothings appear as sections of projective varieties
constructed as higher dimensional unprojections of
$\mathbb{P}^1\times\mathbb{P}^1\times\mathbb{P}^1$ or
$\mathbb{P}^2\times\mathbb{P}^2$.
\end{rem}
\begin{rem}
We can perform a similar construction in other cases. It surely works for other
del Pezzo surfaces, but it might also be successful when the Calabi--Yau
threefolds are not complete intersections. In the latter cases the construction
of the Fano fourfolds extending our Calabi--Yau threefolds is much more
delicate.
However, once we manage to construct the higher dimensional unprojection in such
a way that it ends up in a variety with isolated singularities our smoothing
will be obtained as a general section of the constructed variety. The linkage
construction then, gives us a more or less explicit description of the smoothed
variety by giving its resolution.
\end{rem}
 \end{section}
\begin{section}{Complete intersections in homogeneous spaces}
In this section we consider classical examples of Calabi--Yau
threefolds which are complete intersections in homogeneous
varieties (see \cite{Borc}). We prove that we can construct the
following cascade of unprojections with a del Pezzo surface of degree $4$ at
each step.
\begin{longtable}
 {|c|c|c|c|c|}\hline
            $H^3$&$\chi$&$c_2\cdot H$&$dim|H|$&Description        \\\hline\hline
      \endhead
 4 &-256&52&5&   $X_{2,6}\subset\mathbb{P}(1,1,1,1,1,3)$         \\ \hline
 8 &-176&56&6&   $X_{2,4}\subset\mathbb{P}^5$                   \\ \hline

 12&-144&60&7&   $X_{2,2,3}\subset\mathbb{P}^6$                  \\ \hline

 16&-128&64&8&   $X_{2,2,2,2}\subset\mathbb{P}^7$                \\ \hline

20&-120&68&9&   $X_{1,2,2}\subset G(2,5)$                     \\
\hline

 24&-116&72&10&  $X_{1,1,1,1,1,1,2}\subset S_{10}$               \\ \hline

 28&-116&76&11&  $X_{1,1,1,1,2}\subset G(2,6)$                   \\ \hline

 32&-116&80&12&  $X_{1,1,2}\subset LG(3,6)$                     \\ \hline

  36&-120&84&13&  $X_{1,2}\subset G_2$                            \\ \hline

 \end{longtable}
Here $LG(3,6)$ denotes the Lagrangian Grassmannian,  $S_{10}$ stands for the
orthogonal Grassmannian parameterizing four-dimensional spaces in a quadric of
maximal rank in $\mathbb{P}^9$, and $G_2$ is the variety parameterizing
isotropic
5-spaces of a general 4-form in $\mathbb{C}^7$. The varieties $X_2\subset
G(2,5)$, $S_{10}$, $G(2,6)$, $LG(3,6)$ and $G_2$ are commonly called Mukai
varieties and denoted alternatively by $M_6,\dots ,M_{10}$ respectively.
\begin{rem}
Observe that we saw already three possibilities of unprojecting a del Pezzo
surface of degree 4 in a degeneration of a Calabi--Yau threefold corresponding to the
family  $X_{2,2} \subset G(2,5)$. These were Tom and Jerry and the middle part
of the cascade. However, only the last one leads us to a Calabi--Yau threefold
with Picard group of rank one.
\end{rem}

Up to the case $D_4 \subset X_{1,2,2}\cap G(2,5)$ the unprojections are standard
examples discussed above. To proceed with the other cases we again increase the
dimension of the pictures and use slightly more general theorems that can be
found in \cite{RI,R,Kaskmuk}.
It is proved there that a general nodal proper linear section of a Mukai variety
$M_{g}$ projected from its node is a proper linear section of $M_{g-1}$.  

Having this we construct needed unprojections as sections of these higher
dimensional unprojections by an appropriate number of hyperplanes and a quadric
cone. We construct in this way the whole cascade.
To obtain geometric bitransitions we need to prove that in each case there is an
intermediate Calabi--Yau threefold $Z$ with morphisms 
factorizing the unprojection. It will again just be the blow up of the
singularity of the projected threefold. Its tangent
cone is isomorphic to some generic complete intersection of two quadric cones,
hence 
to a cone over the del Pezzo surface of degree 4. The blow up of this
singularity 
is a crepant resolution and the projection factorizes through it and a
contraction morphism, which can easily be checked to be a small contraction.

\begin{rem} In the same way as in Remark \ref{uwagavs}, the mirrors of the three
examples of degrees 24, 32, and 36 are not known, but have a conjectured
Picard-Fuchs equation. One may try to construct the mirrors using the
constructed geometric transitions.
\end{rem}

\end{section}

\begin{section}*{Acknowledgements}
Most results of the paper were obtained while the author visited the
University of Oslo supported by an EEA Scholarship and Training fund
in Poland and by MNiSW grant N N201 388834. The project is also partially supported by 
SNF, No 200020-119437/1. I would like to thank G. Kapustka,
K. Ranestad and A. Lutken for enlightening discussions and remarks. I would like
also 
to thank the referee for precious comments, major corrections, and suggestions of subjects for future investigation.
\end{section}

\bigskip
\begin{minipage}{10cm}
Institute of Mathematics of the Jagiellonian University, 
ul. \L ojasiewicza 6,\\
30-348 Krak\'ow.\\
michal.kapustka@uj.edu.pl
\end{minipage}
\end{document}